\documentclass{amsart}
\usepackage{amssymb}
\newtheorem{theorem}{Theorem}
\newtheorem{lemma}[theorem]{Lemma}

\newtheorem{corollary}[theorem]{Corollary}

\theoremstyle{definition}

\theoremstyle{remark}

\newcommand{\diam}{\mathrm{diam}}

\newcommand{\conv}{\mathrm{conv}}

\newcommand{\R}{\mathbb{R}}
\newcommand{\N}{\mathbb{N}}

\newcommand{\X}{\mathrm{X}}

\newcommand{\Y}{\mathrm{Y}}
\newcommand{\Z}{\mathrm{Z}}
\newcommand{\B}{\mathbf{B}}

\renewcommand{\S}{\mathbf{S}}

\begin{document}

\title[Superreflexive almost transitive Banach spaces]{A note on the class of superreflexive\\ almost transitive Banach spaces}

\author{Jarno Talponen}
\address{University of Helsinki, Department of Mathematics and Statistics, Box 68, (Gustaf H\"{a}llstr\"{o}minkatu 2b) FI-00014 University
of Helsinki, Finland}
\email{talponen@cc.helsinki.fi}

\date{\today}

\begin{abstract}
The class $\mathcal{J}$ of simultaneously almost transitive, uniformly convex and
uniformly smooth Banach spaces is characterized in terms of convex-transitivity and the weak geometry of the norm.
\end{abstract}

\maketitle

\section*{Introduction}
This note investigates the interplay between the geometry of the norm and isometric symmetries of Banach spaces.
We denote the closed unit ball of a real Banach space $\X$ by $\B_{\X}$ and the unit sphere of $\X$ by $\S_{\X}$.
A Banach space $\X$ is called \emph{transitive} if for each $x\in \S_{\X}$ the corresponding orbit
$\mathcal{G}_{\X}(x)\stackrel{\cdot}{=}\{T(x)|\ T\colon \X\rightarrow \X\ \mathrm{is\ an\ isometric\ automorphism}\}$
coincides with $\S_{\X}$. If $\overline{\mathcal{G}_{\X}(x)}=\S_{\X}$ (respectively $\overline{\conv}(\mathcal{G}_{\X}(x))=\B_{\X}$)
for all $x\in\S_{\X}$, then $\X$ is called \emph{almost transitive} (respectively \emph{convex-transitive}).
We refer to \cite{BR2} for an extensive survey of these and other related concepts.

We will study the class of simultaneously almost transitive, uniformly convex and uniformly smooth Banach spaces.
This class has been studied previously by Finet \cite{Fi}, Cabello \cite{Ca}, Becerra and Rodriguez \cite{BR}
and we denote it by $\mathcal{J}$ following \cite{BR2}. An example of such a space is $L^{p}(0,1)$ for $1<p<\infty$.

We will provide a criterion for a convex-transitive Banach space $\X$ to belong to $\mathcal{J}$.
Based on \cite{Fi,Ca,BR} Becerra and Rodriguez summarized in their survey \cite{BR2} some connections between transitivity
conditions and the geometry of the norm. Especially \cite[Thm. 6.8, Cor. 6.9]{BR2} provide
a long list of conditions on $\X$ equivalent to $\X$ being a member of
$\mathcal{J}$. For example, a convex-transitive space $\X$ belongs to
$\mathcal{J}$ provided that it is an Asplund space or has the Radon-Nikodym Property. Observe that since $L^{\infty}(0,1)$
is convex-transitive (see \cite{Rol} or \cite{BR2}) and contains isometrically all separable Banach spaces,
the convex-transitivity condition per se does not guarantee any good geometric behavior.

In order to characterize the convex-transitive spaces, which belong to $\mathcal{J}$, we impose the following two geometric
conditions. Firstly, we require that there exist relatively weakly open subsets of $\S_{\X}$ of diameter less than $2$. 
Secondly, we impose the condition that the $\omega$-exposed points of $\B_{\X}$ are
weakly dense in $\S_{\X}$. See \cite{Su} for important results regarding $\omega$-exposed points.

\subsection*{Preliminaries}
Throughout this article we will consider \emph{real} Banach spaces denoted by $\X,\Y$ and $\Z$ unless otherwise stated.
For a general introduction to the geometry of the norm see the first chapter of \cite{JL}.
Given $f\in\X^{\ast}$ and $\alpha\in\R$ the corresponding open \emph{slice} of $C\subset\X$ is denoted by
$S(C,f,\alpha)=\{x\in C|f(x)>\alpha\}$. Put $\B(x,\epsilon)=x+\epsilon\B_{\X}$ for $x\in\X$ and $\epsilon>0$.
The dual $\X^{\ast}$ of $\X$ is said to be \emph{convex $\omega^{\ast}$-transitive} if
$\overline{\conv}^{\omega^{\ast}}(\{T^{\ast}(f):\ T\in \mathcal{G}_{\X}\})=\B_{\X^{\ast}}$ for $f\in \S_{\X^{\ast}}$.
For a Banach space $\X$ and $u\in \S_{\X}$, we recall that the \emph{modulus of roughness} at $u$ is given by
\begin{equation}
\eta(\X,u)=\inf_{\delta>0}\left\{\sup\left\{\frac{||u+h||+||u-h||-2}{||h||}:\ h\in\X,\ ||h||\leq \delta\right\}\right\}.
\end{equation}
Note that $0\leq \eta(\X,u)\leq 2$. The norm is Fr\'{e}chet differentiable at $u$ if and only if $\eta(\X,u)=0$,
see \cite[Lemma I.1.3]{DGZ}. The space $\X$ is \emph{extremely rough} if $\eta(\X,u)=2$ for $u\in\S_{\X}$.
We will also require the following results. Recall the geometric fact below, which can be found in \cite[I.1.11]{DGZ}.
\begin{lemma}\label{lemma1}
Let $\X$ be a Banach space and $x^{\ast}\in \S_{\X^{\ast}}$. Then
\[\eta(\X^{\ast},x^{\ast})=\inf\{\diam(S(B_{\X},x^{\ast},\alpha)):\ 0<\alpha<1\}.\]
\end{lemma}
The following characterization of the class $\mathcal{J}$ is crucial in our arguments and it is included in \cite[Theorem 1]{BR0}:
\emph{The space $\X$ is a member of $\mathcal{J}$ if and only if $\X^{\ast}$ is convex $\omega^{\ast}$-transitive
and the norm of $\X^{\ast}$ is not extremely rough.}

Recall the classical \emph{\u Smulyan lemma} (see e.g. \cite[Lemma 8.4]{FHHMPZ}),
which states that the below conditions (i)-(iii) are equivalent:
\begin{enumerate}
\item[(i)]{The norm $||\cdot||$ of a Banach space $\X$ is Gateaux differentiable at $x\in \S_{\X}$.}
\item[(ii)]{For all $(f_{n}),(g_{n})\subset \S_{\X^{\ast}}$ such that
$\lim_{n\rightarrow\infty}f_{n}(x)=\lim_{n\rightarrow\infty}g_{n}(x)=1$ it holds that $f_{n}-g_{n}\stackrel{\omega^{\ast}}{\longrightarrow}0$ as $n\rightarrow\infty$.}
\item[(iii)]{There is unique $f\in \S_{\X^{\ast}}$ such that $f(x)=1$, i.e. $x$ is a smooth point.}
\end{enumerate}

The $\omega$-exposed points are a class of nicely rotund points.
Recall that $x\in \S_{\X}$ is said to be \emph{$\omega$-exposed} (resp. \emph{strongly exposed}) if there is
$f\in\S_{\X^{\ast}},\ f(x)=1,$ such that whenever $(x_{n})\subset\B_{\X}$ is a sequence satisfying $\lim_{n\rightarrow \infty}f(x_{n})=1$
then $x=\omega-\lim_{n\rightarrow\infty}x_{n}$ (resp. $x=\lim_{n\rightarrow\infty}x_{n}$). In such a case $f$ above is called a
\emph{$\omega$-exposing functional} (resp. \emph{strongly exposing functional}) for $x$. Actually, a norm-attaining functional
$f\in\S_{\X^{\ast}}$ is a $\omega$-exposing functional for a (unique) point $x\in\S_{\X}$ if and only if
$f$ is a smooth point in $\X^{\ast}$ (see e.g. \cite{talponen}).

\section*{Results}

The following Theorem 2. gives a criterion for convex-transitive Banach spaces to be members of the class $\mathcal{J}$.
Recall that $\mathcal{J}$ is the class of almost transitive Banach spaces, which are simultaneously uniformly convex and uniformly
smooth.
\begin{theorem}\label{thm1}
Let $\X$ be a convex-transitive Banach space satisfying the following conditions:
\begin{enumerate}
\item[$(a)$]{The weakly exposed points are relatively weakly dense in $\S_{\X}$.}
\item[$(b)$]{$\inf\{\diam(U)|U\subset \S_{\X}\ \mathrm{relatively}\ \omega\mathrm{-open},\ U\neq\emptyset\}<2$.}
\end{enumerate}
Then $\X$ belongs in $\mathcal{J}$.
\end{theorem}
Let us discuss the assumptions before giving the proof.
The assumption (b) above can be considered as a weakening of the Point of Continuity Property (PCP),
of the Radon-Nikodym property or of the Kadec-Klee property. Recall that a Banach space $\X$ is said to have the
Point of Continuity Property if each non-empty subset $A\subset\X$ contains a point $x$ such that for each
$\epsilon>0$ there is a non-empty relatively weakly open neighbourhood $U\subset A$ of $x$ such that $\diam(U)<\epsilon$.
It should be mentioned that PCP in turn is a much weaker property than asymptotical uniform convexity (see \cite{JLPS} for a discussion)
and that rotations have not been studied in connection with these two properties.

Regarding assumption (a), as pointed out in the Preliminaries, a point $x\in\S_{\X}$ is $\omega$-exposed
if it is exposed by a smooth functional $f\in\S_{\X^{\ast}}$. For comparison, recall that the analogous fact holds
for strongly exposed points and Fr\'{e}chet smooth functionals. Moreover, the requirement of a point $x$ being $\omega$-exposed
is much weaker than that of being strongly exposed. These considerations yield the following consequence of
Theorem 2.
\begin{corollary}
If a convex-transitive space $\X$ has the PCP and $\X^{\ast}$ is a Gateaux smooth space, then $\X\in \mathcal{J}$.
\end{corollary}

\begin{proof}[Proof of Theorem 1.]
According to assumption (b) there exists a relatively weakly open set $U\subset \S_{\X}$ such that $\diam(U)<2$.
Assumption (a) yields that there exists a weakly exposed point $z\in U$.
If $f\in\S_{\X^{\ast}}$ is a weakly exposing functional for $z$, then there exists $\alpha\in (0,1)$ such that
$S(\B_{\X},f,\alpha)\cap \S_{\X}\subset U$. Indeed, assume to the contrary that
$S(\B_{\X},f,\alpha)\cap \S_{\X}\setminus U\neq\emptyset$
for all $0<\alpha<1$. Then one can pick a sequence $(v_{k})_{k\in\N}\subset\S_{\X}$ not intersecting $U$ such that
$f(v_{k})\rightarrow 1$ as $k\rightarrow\infty$. Since $f$ is a weakly exposing functional for $z$,
we get that $v_{k}\stackrel{\omega}{\longrightarrow}z$ as $k\rightarrow\infty$.
This provides a contradiction since $U$ is a $\omega$-open neighbourhood of $z$.

Let $\delta\in (\alpha,1)$ be such that $\diam(U)+2(1-\delta)<2$. Then we get that
\begin{equation}
\diam(S(\B_{\X},f,\delta))\leq \diam(S(\B_{\X},f,\delta)\cap \S_{\X})+2(1-\delta)<2,
\end{equation}
since $S(\B_{\X},f,\delta)\cap \S_{\X}\subset U$. Hence we obtain by Lemma \ref{lemma1} that $\X^{\ast}$ is not extremely rough.
Since $\X$ is convex-transitive, we obtain by a simple and well-known argument that $\X^{\ast}$ is
convex $\omega^{\ast}$-transitive. Hence $\X$ is a member of $\mathcal{J}$ according to \cite[Theorem 1]{BR0}.
\end{proof}

\begin{theorem}\label{thm2}
Let $\X$ be a convex-transitive Banach space. Suppose that for each $\epsilon>0$ there exist $k\in \N$,
$\delta_{1},\ldots,\delta_{k}\in (0,1)$, $c_{1},\ldots, c_{k}\in (0,1]$ with $\sum_{i} c_{i}=1$ 
and $f_{1},\ldots,f_{k}\in \S_{\X^{\ast}}$ such that the following conditions hold:
\begin{enumerate}
\item[(1)]{$\{y\in\B_{\X}|f_{i}(y)> 1-c_{i}\delta_{i}\ \mathrm{for}\ 1\leq i\leq k\}\neq \emptyset$,} 
\item[(2)]{$\diam(\{y\in\B_{\X}|f_{j}(y)> 1-2\delta_{i}\ \mathrm{for}\ 1\leq i\leq k\})<\epsilon$.}
\end{enumerate}
Then $\X$ belongs in $\mathcal{J}$.
\end{theorem}
\begin{proof}
We consider $\X$ embedded canonically in $\X^{\ast\ast}$.
Fix $0<\epsilon<1$, $x\in \S_{\X}$ and $u^{\ast\ast},v^{\ast\ast}\in \S_{\X^{\ast\ast}}$ such that 
$\frac{u^{\ast\ast}+v^{\ast\ast}}{2}=x$.

As in the assumptions, let $(\delta_{i})$, $(c_{i})$ and $(f_{i})$ be $k$-uples such that
\begin{equation}
\{y\in \B_{\X}:\ f_{i}(y)> 1-c_{i}\delta_{i}\ \mathrm{for}\ 1\leq i\leq k\}\neq\emptyset
\end{equation}
and 
\begin{equation}
V\stackrel{\cdot}{=}\{y\in \B_{\X}:\ f_{i}(y)> 1-2\delta_{i}\ \mathrm{for}\ 1\leq i\leq k\}\subset \B_{\X}
\end{equation}
satisfies $\diam(V)<\epsilon$. 
Define $W=\{y\in \B_{\X^{\ast\ast}}:\ f_{i}(y)> 1-2\delta_{i}\ \mathrm{for}\ 1\leq i\leq k\}$.

Recall that $\B_{\X}$ is $\omega^{\ast}$-dense in $\B_{\X^{\ast\ast}}$ according to Goldstine's theorem.
Since $W$ is relatively $\omega^{\ast}$-open in $\B_{\X^{\ast\ast}}$, we obtain that $V$ is $\omega^{\ast}$-dense in $W$.
Observe that 
$\{y-z:\ y,z\in V\}$ is $\omega^{\ast}$-dense in
$\{u^{\ast\ast}-v^{\ast\ast}:\ u^{\ast\ast},v^{\ast\ast}\in W\}$. Thus, by using the
$\omega^{\ast}$-lower semicontinuity of the norm of $\X^{\ast\ast}$ we obtain that
$\diam(W)=\diam(V)<\epsilon$.

Pick $w\in \B_{\X}$ such that $f_{i}(w)> 1-c_{i}\delta_{i}$ for $1\leq i\leq k$.
Observe that $w\in\overline{\conv}(\mathcal{G}(x))$ as $\X$ is convex-transitive.
Let $\{x_{n}|n\in\N\}\subset\mathcal{G}(x)$ be such that $w\in \overline{\conv}(\{x_{n}|n\in\N\})$.
Fix a sequence $(a^{(j)})_{j\in \N}\subset \S_{\ell^{1}}$ of coordinate-wise non-negative vectors such that 
\[\left|\left|\sum_{n\in \N}a^{(j)}_{n}x_{n}-w\right|\right|\rightarrow 0\ \mathrm{as}\ j\rightarrow \infty\]
and the continuity of $f_{i}$ for $1\leq i\leq k$ yields that 
\[f_{i}\left(\sum_{n\in\N} a^{(j)}_{n}x_{n}\right)\rightarrow f_{i}(w)\ \mathrm{as}\ j\rightarrow \infty\ \mathrm{for}\ 1\leq i\leq k.\]
By the selection of $w$ we obtain that  
\begin{equation}\label{eq: Kij}
\limsup_{j\rightarrow \infty}\sum_{n\in \N\setminus K_{i}} a^{(j)}_{n}< c_{i},\quad 1\leq i\leq k,
\end{equation}
where $K_{i}=\{n\in \N|\ f_{i}(x_{n})>1-\delta_{i}\}$.
Indeed, since $||f_{i}||=1$ for $1\leq i\leq k$, we get
\[\sum_{n\in\N} a^{(j)}_{n}f_{i}(x_{n})\leq \sum_{n\in K_{i}}a_{n}^{(j)}+(1-\delta_{i})\sum_{n\in\N\setminus K_{i}}a_{n}^{(j)}\]
for $1\leq i\leq k,\ j\in \N$. Note that $1-t+(1-\delta_{i})t=1-t\delta_{i}$ for $t\in \R$ and $f_{i}(w)>1-c_{i}\delta_{i}$, 
so that placing $t=\sum_{n\in\N\setminus K_{i}}a_{n}^{(j)}$ yields \eqref{eq: Kij}.

Hence we deduce from \eqref{eq: Kij} and the fact that $\sum_{i} c_{i}=1$ that 
\[\liminf_{j\rightarrow \infty}\sum_{n\in K}a_{n}^{(j)}\geq 1-\sum_{i}\ \limsup_{j\rightarrow \infty}\sum_{n\in\N\setminus K_{i}}a_{n}^{(j)}>0,\]
where $K=\bigcap_{1\leq i\leq k}K_{i}$. Consequently $K\neq \emptyset$.
In particular, there exists $T\in \mathcal{G}_{\X}$ satisfying $f_{i}(T(x))> 1-\delta_{i}$ for $1\leq i\leq k$.
Note that $T^{\ast\ast}(x)=\frac{T^{\ast\ast}(u^{\ast\ast})+T^{\ast\ast}(v^{\ast\ast})}{2}$.
Since $1-\delta_{i}<T^{\ast\ast}(x)(f_{i})$ and $\max(T^{\ast\ast}(u^{\ast\ast})(f_{i}),T^{\ast\ast}(v^{\ast\ast})(f_{i}))\leq 1$, 
it follows that $\min(T^{\ast\ast}(u^{\ast\ast})(f_{i}),T^{\ast\ast}(v^{\ast\ast})(f_{i}))>1-2\delta_{i}$ 
for $1\leq i\leq k$. Thus $T^{\ast\ast}(u^{\ast\ast}),T^{\ast\ast}(v^{\ast\ast})\in W$. Hence
\[||u^{\ast\ast}-v^{\ast\ast}||=||T^{\ast\ast}(u^{\ast\ast})-T^{\ast\ast}(v^{\ast\ast})||<\epsilon.\]
Since $u^{\ast\ast},v^{\ast\ast}$ and $\epsilon$ were arbitrary, we get that $x$ is an extreme point of $\B_{\X^{\ast\ast}}$.

Then $T^{\ast\ast}(x)$ is also an extreme point and according to Choquet's Lemma (see e.g. \cite[Lemma 3.40]{FHHMPZ})
the $\omega^{\ast}$-slices of $\B_{\X^{\ast\ast}}$ containing $T^{\ast\ast}(x)$ form a $\omega^{\ast}$-neighbourhood basis
of $T^{\ast\ast}(x)$ in $\B_{\X^{\ast\ast}}$.
Thus there exists $\alpha\in (0,1)$ and $g\in \S_{\X^{\ast}}$ such that $S(\B_{\X^{\ast\ast}},g,\alpha)\subset W$.
In particular $\diam(S(\B_{\X},g,\alpha))<\epsilon<2$. By Lemma \ref{lemma1} we get that $\X^{\ast}$ is not extremely rough.
Since $\X$ is convex-transitive we have that $\X^{\ast}$ is convex $\omega^{\ast}$-transitive.
We conclude by applying \cite[Theorem 1]{BR0} that $\X$ is a member of $\mathcal{J}$.
\end{proof}

We note that assumption (1) in the above result can be replaced by the following stronger condition, which is perhaps
more appealing; namely that there exists $x\in \S_{\X}$ such that 
\[f_{1}(x)=f_{2}(x)=\dots =f_{k}(x)=1.\]

\end{document}